\documentclass[a4paper,twoside,12pt]{article}
\usepackage{amsthm}
\usepackage[a4paper,margin=3cm,centering,nohead,nofoot]{geometry}
\makeatletter \renewenvironment{proof}[1][\proofname]{\par\pushQED{\qed}\normalfont\topsep6\p@\@plus6\p@\relax\trivlist\item[\hskip\labelsep\bfseries#1\@addpunct{.}]\ignorespaces}{\popQED\endtrivlist\@endpefalse} \makeatother
\newtheorem*{theorem}{Theorem}

\begin{document}

\title{A theorem with constructive\\and non-constructive proofs}
\author{Jaime Gaspar\thanks{INRIA Paris-Rocquencourt, $\pi r^2$, Univ Paris Diderot, Sorbonne Paris Cit\'e, F-78153 Le Chesnay, France. \texttt{mail@jaimegaspar.com}, \texttt{www.jaimegaspar.com}. Financially supported by the French Fondation Sciences Math\'ematiques de Paris.}}
\date{5 November 2012}
\maketitle
\thispagestyle{empty}

\begin{abstract}
  We present a very simple example of a theorem with constructive and non-constructive proofs: the equation $c^2 x^2 - (c^2 + c)x + c = 0$ has a solution.
\end{abstract}

A constructive (non-constructive) proof shows the existence of an object by presenting (respectively, without presenting) the object. From a logical point of view, a constructive (non-constructive) proof does not use (respectively, uses) the law of excluded middle.

The discussion of constructive versus non-constructive proofs is very common in mathematical logic and philosophy of mathematics. To illustrate this discussion, it is convenient to have some \emph{very simple} examples of theorems with both constructive and non-constructive proofs. Unfortunately, there seems to be a shortage of such examples. We present here a new example.

\begin{theorem}
  Let $c$ be an arbitrary real constant. The equation $c^2 x^2 - (c^2 + c)x + c = 0$ in $x$ has a solution.
\end{theorem}

\begin{proof}[Non-constructive proof.]
  By the law of excluded middle, we have $c = 0$ or $c \neq 0$.
  \begin{itemize}
    \item Case $c = 0$: $x = 0$ (or any $x$) is a solution.
    \item Case $c \neq 0$: $x = 1/c$ is a solution.
  \end{itemize}
  (This proof is non-constructive because it does not present a solution since it does not decide between the two cases as the equality $c = 0$ is undecidable.)
\end{proof}

\begin{proof}[Constructive proof.]
  We have that $x = 1$ is a solution. (This proof is constructive because it presents a solution.)
\end{proof}

\end{document}